\documentclass[reqno, 11pt]{amsart}
\usepackage{amsmath}
\usepackage{amssymb}
\usepackage{amsmath}
\usepackage{amssymb}
\usepackage{amssymb}
\usepackage{mathrsfs}
\usepackage{amsmath}
\usepackage{amsmath,amsthm,amssymb,amscd}
\usepackage{latexsym}
\usepackage{color}
\usepackage[colorlinks=true]{hyperref}
\hypersetup{urlcolor=blue, citecolor=blue,linkcolor=blue}
\newtheorem{theorem}{Theorem}[section]

\newtheorem{lemma}[theorem]{Lemma}

\newtheorem{remark}[theorem]{Remark}

\theoremstyle{definition} \theoremstyle{remark}
\numberwithin{equation}{section}

\def\R{\mathbb R}

\def\N{\mathbb N}

\def\<{\langle}
\def\>{\rangle}

\def\diag{\mathrm{Diag}}

\newcommand{\tr}{\operatorname{tr}}

\begin{document}

\title[Trace and determinant preserving maps of matrices]{Trace and determinant preserving maps of matrices}

\author{Huajun Huang, Chih-Neng Liu, Patr\'{i}cia Szokol, Ming-Cheng Tsai$^\dag$, Jun Zhang}

\address[Huang and Tsai]{Department of Mathematics and Statistics, Auburn University,
Auburn, AL 36849, USA.}

\email[Huang]{huanghu@auburn.edu}

\email[Tsai]{mctsai2@gmail.com}

\address[Liu]{Department of Applied Mathematics, National Sun Yat-sen University, Kaohsiung, 80424, Taiwan}

\email[Liu]{cnliu@mail.nsysu.edu.tw}

\address[Szokol]{MTA-DE ``Lend\"ulet'' Functional Analysis Research Group, Institute of Mathematics,
University of Debrecen, P.O. Box 12, 4010 Debrecen, Hungary \\
Institute of Mathematics, University of Miskolc, 3515 Miskolc, Hungary}

\email[Szokol]{matszp@uni-miskolc.hu, szokolp@science.unideb.hu}

\address[Zhang]{School of Mathematics and Statistics,
Central China Normal University, Wuhan, Hubei 430079, China.}
\email[Zhang]{zhjun@mail.ccnu.edu.cn}

\thanks{$\dag$ corresponding author}

 \date{\today}

\maketitle

\begin{abstract}

Suppose a map $\phi$ on the set of positive definite matrices satisfies $\det(A+B)=\det(\phi(A)+\phi(B))$. Then we have $$\tr(AB^{-1}) =  \tr(\phi(A){\phi(B)}^{-1}).$$
Through this viewpoint, we show that $\phi$ is of the form $\phi(A)= M^*AM$ or $\phi(A)= M^*A^tM$ for some invertible matrix $M$ with $\det (M^*M)=1$. We also characterize the map $\phi: \mathcal{S} \rightarrow \mathcal{S}$ preserving the determinant of convex combinations in $\mathcal{S}$ by using similar method. Here $\mathcal{S}$ can be the set of complex matrices, positive definite matrices, symmetric matrices, and upper triangular matrices.
\vspace{0.4cm}

\end{abstract}

\section{Introduction}

Let $M_n$ be the set of $n\times n$ complex matrices. Let $H_n$ (resp., $M_n^{+}$, $P_n$, $S_n$, $T_n$, $D_n$) be the set of hermitian (resp., positive semi-definite, positive definite, symmetric, upper triangular, diagonal) matrices.

Over the last years, one of the most active research topics in matrix theory is the linear preserver problem (see \cite{Br97}, \cite{Li01}, \cite{Li92}). Many interesting results about preserver problems in different matrices were discussed and obtained. They included preservers on determinant, eigenvalue, spectrum, permanent, rank, commutativity, product, trace, norm,  etc. In 1897, Frobenius  studied the following problem. Let $\phi: M_n\rightarrow M_n$ be a linear mapping satisfying
\begin{equation}\label{97}
\det(\phi(A))=\det(A), \quad \quad A \in M_n.
\end{equation}
Then there exist $M$, $N\in M_n$ with $\det(MN)=1$ such that either
\begin{equation}\label{Mnform1}
\phi(A)=MAN, \quad \quad A \in M_n,
\end{equation}
or
\begin{equation}\label{Mnform2}
\phi(A)=MA^tN, \quad \quad A \in M_n
\end{equation}
(see \cite{Fr97}). Here $A^t$ denotes the transpose of $A\in M_n$. In 1959, Marcus and Moyls (\cite{Ma59}) proved that if $\phi: M_n\rightarrow M_n$ is a linear map preserving the set of matrices of rank $k$ for all $k=1, \ldots, n$, then there exist invertible matrices $U$ and $V$ in $M_n$ such that  $\phi(A)=UAV$ for all $A$ in $M_n$ or $\phi(A)=UA^tV$ for all $A$ in $M_n$. Applying this result, they gave another proof of the problem (\ref{97}).
In 1969, Eaton showed that a linear map $\phi$ on the linear space of $n\times n$ real symmetric matrices sending the elements in the cone of $n\times n$ real positive definite matrices to the same cone and satisfying $\det(\phi(A))=c(\det A)$ for a nonzero real constant $c$ must be of the form $\phi(A)=MAM^t$ for some invertible matrix $M$ in $M_n$ (see \cite{Ea69}).
In 2002 and 2003, Dolinar and \v{S}emrl (\cite{Do02}) and Tan and Wang (\cite{Tan03}) considered Frobenius problem by removing the linearity of $\phi$ and changing condition (1.1) to
\begin{equation}\label{02}
\det(\phi(A)+\lambda\phi(B))=\det(A+\lambda B), \quad \quad A \in M_n, \lambda\in \mathbb{C}.
\end{equation}
Then they showed that the {conclusion is the same as that of the Frobenius Theorem, i.e. $\phi$ is either of the form (1.2) or (1.3)}. Moreover, in \cite{Tan03} Tan and Wang also considered the problem on $T_n$. In 2004, Cao and Tang (\cite{Ca04}) studied the problem on $S_n$.

Alternatively, Wigner's unitary-antiunitary theorem says that if $\phi$ is a bijective map defined on the set of all rank one projections acting on a Hilbert space $H$ {satisfying} $\tr(\phi(P)\phi(Q))=\tr(PQ)$, then there {exists} a unitary operator $U$ on $H$ such that $\phi(P)=U^{*}PU$ or $\phi(P)=U^{*}P^{t}U$ for all rank one projections $P$. In 1963, Uhlhorn {generalized} Wigner's theorem to show that the same conclusion holds if the equality $\tr(\phi(P)\phi(Q))=\tr(PQ)$ is replaced by $\tr(\phi(P)\phi(Q))=0 \Leftrightarrow \tr(PQ)=0$ (see \cite{Uh63}). In 2012, Li, Plevnik, and {\v S}emrl (\cite{Li12}) characterized bijective maps $\phi: S\rightarrow S$ satisfying $\tr(\phi(A)\phi(B))=c  \Leftrightarrow \tr(AB)=c$ for a given real number $c$, where $S$ is the set of $n\times n$ hermitian matrices, the set of $n\times n$ real symmetric matrices, or the set of projections of rank one.

In this paper, we study the relationship {between} the maps preserving determinants and trace equalities, and characterize this kind of maps. The paper is organized as follows. In Section 2, we show that a map $\phi$ on $\mathcal{S}$ satisfying trace equalities $\tr(\phi(A)\phi(B)^{-1})=\tr(AB^{-1})$ or $\tr(\phi(A)\phi(B))=\tr(AB)$ is linear, where $\mathcal{S}=S_n$, $M_n$, $T_n$, or $P_n$. These results are useful in the consequent sections. In Section 3, we consider the maps $\phi: P_n\rightarrow P_n$ satisfying
\begin{equation}\label{15}
\det(\phi(A)+\phi(B))=\det(A+B), \quad \quad A, B \in P_n.
\end{equation}
Through differentiation, we can get $\tr(\phi(A){\phi(B)}^{-1})=\tr(AB^{-1})$. Then we show that $\phi$ is of the form $\phi(A)= M^*AM$ or $\phi(A)= M^*A^tM$ for some invertible matrix $M$ with $\det (M^*M)=1$. In Sections 4, 5, and 6, we consider the analogous problems on $S_n$, $M_n$, and $T_n$, and provide the similar theorems to enrich Dolinar, {\v S}emrl, Tan, Wang, Cao, and Tang's results (see \cite{Ca04}, \cite{Do02} and \cite{Tan03}).

\section{Preliminaries}

In this section, we provide some lemmas to show that some trace equalities could lead to linearity.


\begin{lemma}\label{homo1}
Let $\mathcal{A}$ and $\mathcal{C}$ be subspaces of $M_n$, $\mathcal{B}$, $\mathcal{D}$ and $\mathcal{E}$ be subsets of $M_n$, $\phi: \mathcal{A}\rightarrow \mathcal{C}$, $f: \mathcal{B}\rightarrow \mathcal{D}$, and $g: \mathcal{B}\rightarrow \mathcal{E}$ be maps that satisfy the following conditions.
\begin{itemize}
\item[(1)] $\tr(\phi(A)f(B))=\tr(Ag(B)), \quad \quad \forall \ A\in\mathcal{A}, \ B\in \mathcal{B}.$
\item[(2)] Given $C\in \mathcal{C}$, if $\tr(C f(B))=0$ for all $B \in \mathcal{B}$, then $C=0$.
\end{itemize}
Then $\phi$ is linear.

\end{lemma}

\begin{proof}
For any $A\in \mathcal{A}$, $B\in \mathcal{B}$, $\lambda\in \mathbb{C}$, by condition (1), we have $$\tr[(\phi(\lambda A)-\lambda\phi(A))f(B)]
=\tr(\lambda Ag(B))-\lambda\tr(Ag(B))=0.$$
Hence $\phi(\lambda A)=\lambda\phi(A)$ by condition (2). Similarly, we can obtain that $\phi(A_1+A_2)=\phi(A_1)+\phi(A_2)$  \qedhere

\end{proof}

The following lemma will be used to prove Lemmas \ref{homo} and \ref{homoPn}.

\begin{lemma}\label{lma1}
Let $A\in \mathcal{S}$, $\mathcal{S}=M_n, H_n, S_n, D_n$. If  $\tr(AB)=0$ for any invertible matrix $B$ in $\mathcal{S}$, then $A=0$.
 \end{lemma}

\begin{proof}

Suppose $A=(a_{jk})\in M_n$. Let $B=\lambda I_n+E_{kj}$, {where $\lambda\neq0, -1$, $\lambda\in \mathbb{R}$, and $E_{kj}$ denotes the matrix having a 1 in the $(k, j)$-th position and zeros elsewhere, $1\leq j, k\leq n$}. Then $0=\tr(AB)=\lambda \tr(A)+a_{jk}$ and hence $a_{jk}=0$ for all $1\leq j, k\leq n$. Thus $A=0$.

Let $A=(a_{jk})\in\mathcal{S}$. Similarly, if $\mathcal{S}=S_n$, then we can choose $B=\lambda I_n+E_{jk}+E_{kj}$, where $\lambda\neq0, \pm1, -2$, $1\leq j, k\leq n$ to obtain $A=0$. In addition, if $\mathcal{S}=D_n$, then we can choose $B=\lambda I_n+E_{jj}$, where $\lambda\neq0, 1$, $1\leq j\leq n$ to derive $A=0$. Finally, if $\mathcal{S}=H_n$, then we can choose $B=\lambda I_n+E_{jk}+E_{kj}$, where $\lambda\neq0, \pm1, -2$, $1\leq j, k\leq n$ to obtain $a_{jk}+a_{kj}=0$. On the other hand, let $B=\lambda I_n+iE_{jk}+E_{kj}$, where $\lambda\neq0, -1$, $\lambda\in \mathbb{R}$, $1\leq j\neq k\leq n$ to derive $a_{jk}+ia_{kj}=0$. Hence we can deduce that $A=0$.  \qedhere

%
%
%
%
%
%
%
%
%
%
%
%
%
%
%
%


 \end{proof}

Now we present Lemmas \ref{homo} and \ref{homoPn} which will be applied {in} the following sections. Firstly, we consider the cases $M_n$, $S_n$ and $T_n$.

\begin{lemma}\label{homo}

Let $(\mathcal{A}, \mathcal{C})$ be $(M_n, M_n)$, $(S_n, S_n)$, or $(T_n, D_n)$. Let $\phi: \mathcal{A}\rightarrow \mathcal{C}$ be a map sending invertible matrices in $\mathcal{A}$ into invertible matrices in $\mathcal{C}$. Suppose $\phi$ satisfy one of the following conditions.
\begin{itemize}
\item[(1)] $\tr(\phi(A)\cdot {\phi(B)}^{-1})=\tr(A\cdot B^{-1}), \quad \quad \forall \ A\in\mathcal{A}, \ \text{invertible \ matrix} \ B\in \mathcal{A}.$
\item[(2)] $\tr(\phi(A)\cdot \phi(B))=\tr(A\cdot B), \quad \quad \forall \ A\in\mathcal{A}, \ \text{invertible \ matrix} \ B\in \mathcal{A}.$
\end{itemize}
Then $\phi$ is linear. Moreover, $\phi$ is bijective if $\mathcal{A}=M_n$ or $S_n$.

%

\end{lemma}

\begin{proof}

(1) Suppose $\phi$ satisfies $\tr(A\cdot B^{-1}) =  \tr(\phi(A)\cdot {\phi(B)}^{-1})$.

$\mathcal{A}=M_n$ or $S_n$:  First of all, select a basis $\{B_1, B_2, \ldots, B_{N}\}$ of invertible matrices in $\mathcal{A}$, where $N=\dim \mathcal{A}$. Then we claim that $\{\phi(B_j^{-1})^{-1}: 1\leq j\leq N\}$ are linearly independent in $\mathcal{A}$. Suppose $\sum_{1\leq j\leq N}\lambda_j\phi(B_j^{-1})^{-1}=0$. Hence for any $A\in \mathcal{A}$
\begin{align*}
\tr A(\sum_{1\leq j\leq N}\lambda_{j}(B_{j}^{-1})^{-1})   &   = \sum_{1\leq j\leq N}\lambda_{j}\tr(A(B_{j}^{-1})^{-1})   \\
& = \sum_{1\leq j\leq N}\lambda_{j}\tr(\phi(A)\phi(B_{j}^{-1})^{-1})=0.
\end{align*}
By Lemma \ref{lma1}, $\sum_{1\leq j\leq N}\lambda_{j}B_{j}=0$ and hence all $\lambda_j=0$. Thus { the set $\{\phi(B_{j}^{-1})^{-1}: 1\leq j\leq N\}$ is linearly independent and spans $\mathcal{A}$}. Suppose $C\in \mathcal{A}$, for any invertible matrix $A\in \mathcal{A}$, $\tr(C\phi(A)^{-1})=0$. Then for any { invertible } $B\in \mathcal{A}$, $\tr(CB)=0$. By Lemma \ref{lma1} again, $C=0$. Therefore, $\phi$ is linear by Lemma \ref{homo1} with $f(B)=\phi(B)^{-1}$ and $g(B)=B^{-1}$.

Now we want to prove that $\phi$ is injective. Assume $\phi(A)=0$ for some $A\in \mathcal{A}$. Then for any invertible $B$ in $\mathcal{A}$, $\tr(AB^{-1})=\tr(\phi(A)\phi(B)^{-1})=0$. By Lemma \ref{lma1}, we have $A=0$. Thus $\phi$ is injective and hence $\phi$ is bijective.

$\mathcal{A}=T_n$: Similarly, we can select an invertible basis $\{B_1, B_2, \ldots, B_{n}\}$ of $D_n$, and using Lemmas \ref{homo1} and \ref{lma1} to prove that $\phi$ is linear.

(2) Suppose $\phi$ satisfies $\tr(A\cdot B) =  \tr(\phi(A)\cdot \phi(B))$. Similarly, let $f(B)=\phi(B)$ and $g(B)=B$, we can apply Lemma \ref{homo1} with Lemma \ref{lma1} to prove that $\phi$ is linear.  \qedhere

%
%
%

 \end{proof}

%
%
%
%

For the case $P_n$, we have an analogous result.

\begin{lemma}\label{homoPn}

Let $\phi: P_n\rightarrow P_n$ be a map. Suppose $\phi$ satisfies one of the following conditions.
\begin{itemize}
\item[(1)] $\tr(\phi(A)\cdot {\phi(B)}^{-1})=\tr(A\cdot B^{-1}), \quad \quad \forall \ A, B\in P_n.$
\item[(2)] $\tr(\phi(A)\cdot \phi(B))=\tr(A\cdot B), \quad \quad \forall \ A, B\in P_n.$
\end{itemize}
Then $\phi$ is injective, additive and $\phi$ satisfies $\phi(\lambda A)=\lambda \phi(A)$ for all $A\in P_n$, $\lambda \in \R^{+}$.

\end{lemma}

\begin{proof}

(1) Suppose $\phi$ satisfies $\tr(\phi(A)\cdot {\phi(B)}^{-1})=\tr(A\cdot B^{-1})$. We extend $\phi$ to a map $\widetilde{\phi}: H_n\rightarrow H_n$ that satisfies $\tr(\widetilde{\phi}(A)\cdot {\phi(B)}^{-1})=\tr(A\cdot B^{-1})$ for $A\in H_n$, $B\in P_n$. Fix a maximal real linearly independent subset $\{S_i: 1\leq i\leq n^2\}$ of $P_n$. Then $\{S_i: 1\leq i\leq n^2\}$ is a real basis of $H_n$. For $A\in P_n$, we define $\widetilde{\phi}(A)=\phi(A)$. For $A\in H_n \setminus P_n$, suppose $A=\sum\limits_{i=1}^{n^2}c_iS_{i}$, we define $\widetilde{\phi}(A)=\sum\limits_{i=1}^{n^2}c_i\phi(S_{i})$. Then $\widetilde{\phi}$ satisfies the desired property. In addition, suppose $C\in H_n$ with $\tr(CB)=0$ for any $B\in P_n$. Then for any $\widetilde{B}\in H_n$, it can be expressed as $\widetilde{B}=B_1-B_2$ for some $B_1, B_2\in P_n$. Thus $\tr(C\widetilde{B})=0$ and hence $C=0$ by Lemma \ref{lma1}. By similar argument in the proof of Lemma \ref{homo}, we can show that $\widetilde{\phi}$ is linear and injective. Therefore, $\phi$ is injective, additive and satisfies $\phi(\lambda A)=\lambda \phi(A)$ for any $A\in P_n$ and $\lambda\in R^{+}$ by Lemmas \ref{homo1} and \ref{lma1}.


(2) Suppose $\phi$ satisfies $\tr(A\cdot B) =  \tr(\phi(A)\cdot \phi(B))$. Then our assertion follows from the similar argument in above case (1).  \qedhere  \end{proof}

\section{Maps  preserving determinants of convex combinations on $P_n$}

The aim of the present section is to prove the following theorem.

\begin{theorem}\label{detP1}

Let $\phi: P_n\rightarrow P_n$ be a map and $\alpha=\det \phi(I)^{1/n}$. Then the following statements are equivalent.
\begin{itemize}
\item[(1)] There exists an invertible matrix $M\in M_n$ with $\det (M^*M)=1$ such that $\phi$ is either
\begin{equation*}
\phi(A)=\alpha M^{*}AM   \quad  \quad  \forall \ A\in P_n \quad    \quad \text{or}  \quad    \quad    \phi(A)=\alpha M^{*}A^{t}M  \quad \quad \forall \ A\in P_n.
\end{equation*}
\item[(2)]
\begin{equation}\label{1}
\det(\phi(A)+\phi(B))=\alpha^n\det(A+B), \quad \quad A,B \in P_n.
\end{equation}
\item[(3)] $  \tr(\phi(A){\phi(B)}^{-1})=\tr(AB^{-1}), \quad \quad A,B \in P_n.$
\end{itemize}

\end{theorem}

%
%
%
%
%
%
%
%
%
%
%

%
%
%
%
%
%
%
%
%
%
%
%

\begin{proof} It is easy to see that (1) $\Rightarrow$ (2) and (1) $\Rightarrow$ (3).

(2) $\Rightarrow$ (1): 
Define $\psi: P_n \rightarrow P_n$ by $\psi(A)=\phi(I)^{-1/2}\phi(A)\phi(I)^{-1/2}$. Then we can see that $\psi$ is unital and $\det(\psi(A)+\psi(B))=\frac{1}{\det\phi(I)}\det(\phi(A)+\phi(B))=\det(A+B)$. Hence by (2) $\Rightarrow$ (1) of Lemma \ref{traceAB3} (see below), there exists a unitary matrix $U\in M_n$ such that $\psi$ is either $\psi(A)=U^{*}AU$ or $\psi(A)=U^{*}A^{t}U$. Hence $\phi$ is either $\phi(A)=\alpha M^{*}AM$ or $\phi(A)=\alpha M^{*}A^{t}M$, where $M=\alpha^{-1/2} U\phi(I)^{1/2}$ and $\det(M^*M)=\det(\alpha^{-1}\phi(I))=\alpha^{-n}\det(\phi(I))=1$.

(3) $\Rightarrow$ (1):  Define $\psi: P_n \rightarrow P_n$ by $\psi(A)=\phi(I)^{-1/2}\phi(A)\phi(I)^{-1/2}$. Then we can see that $\psi$ is unital and
\begin{align*}
\tr(\psi(A)\cdot {\psi(B)}^{-1}) = & \tr( \phi(I)^{-1/2}\phi(A)\phi(B)^{-1}\phi(I)^{1/2})\\
= & \tr( \phi(A)\phi(B)^{-1}\phi(I)^{1/2}\phi(I)^{-1/2})\\
= & \tr(\phi(A)\phi(B)^{-1})=\tr(AB^{-1})
\end{align*}
for $A,B \in P_n$. Hence by Lemma \ref{traceAB3} again, there exists a unitary matrix $U\in M_n$ such that $\psi$ is either $\psi(A)=U^{*}AU$ or $\psi(A)=U^{*}A^{t}U$. Similar to the proof (2) $\Rightarrow$ (1), our assertion follows.  \qedhere  \end{proof}

Now we present  Lemma \ref{traceAB3}, which plays an important role in the proof of Theorem \ref{detP1}.

\begin{lemma}\label{traceAB3}
Let $\phi: P_n\rightarrow P_n$ be a unital map. Then the following statements are equivalent.
\begin{itemize}
\item[(1)] There exists a unitary matrix $U\in M_n$ such that $\phi$ is either
\begin{equation*}
\phi(A)=U^{*}AU   \quad  \quad  \forall \ A\in P_n \quad    \quad  \text{or}  \quad    \quad    \phi(A)=U^{*}A^{t}U  \quad \quad \forall \ A\in P_n.
\end{equation*}
\item[(2)] $\det(\phi(A)+\phi(B))=\det(A+B), \quad \quad A,B \in P_n.$
\item[(3)] $\tr(\phi(A){\phi(B)}^{-1})=\tr(AB^{-1}), \quad \quad A,B \in P_n.$
\item[(4)] $\tr(\phi(A)\phi(B))=\tr(AB), \quad \quad A,B \in P_n.$
\end{itemize}

%
%

\end{lemma}

To prove Lemma \ref{traceAB3} we need some auxiliary lemmas. Firstly, we need the well-known Minkowski's determinant
inequality ({\cite[Theorem 7.8.8]{Horn85}}).

\begin{lemma}\label{Min}
Let $A$, $B\in P_n$. Then
\begin{equation*}
[\det(A+B)]^{1/n}\geq (\det A)^{1/n}+(\det B)^{1/n}, \quad \quad A,B \in P_n,
\end{equation*}
and the equality holds if and only if $B=\lambda A$ for some $\lambda>0$.
\end{lemma}

Using Lemma \ref{Min} we can show that every transformation that preserves the determinant of the convex combination
with parameter $t=1/2$ also preserves it with an arbitrary parameter.

\begin{lemma}\label{equality}
Let $\phi: P_n\rightarrow P_n$ be a map such that
\begin{equation*}
\det(\phi(A)+\phi(B))=\det(A+B), \quad \quad A,B \in P_n.
\end{equation*}
Then $\phi(\lambda A)=\lambda\phi(A)$ for all $A\in P_n$, $\lambda >0$ and
\begin{equation*}
\det(t\phi(A)+(1-t)\phi(B))=\det(tA+(1-t)B) \quad \quad A,B \in P_n, \ 0\leq t\leq 1.
\end{equation*}
\end{lemma}

\begin{proof}
Taking $A=B$,  one has $\det A=\det \phi(A)$. Hence by assumption, we have
\begin{align*}
   [\det(\phi(A)+\phi(\lambda A))]^{1/n}   &   = [\det(A+\lambda A)]^{1/n}\\
   & = (1+\lambda)[\det A]^{1/n}= [\det A]^{1/n}+ [\det\lambda A]^{1/n}\\
   & = [\det \phi(A)]^{1/n}+[\det \phi(\lambda A)]^{1/n}
\end{align*}
for any $\lambda>0$. Thus by Lemma \ref{Min}, we can get $\phi(\lambda A)=a \phi(A)$ for some $a>0$. Since $\det A=\det \phi(A)$, we obtain that
$$\lambda^n \det A = \det (\lambda A)= \det (\phi(\lambda A))=\det (a\phi(A))=a^n \det A$$
Thus $a=\lambda$ and hence $\phi$ satisfies $\phi(\lambda A)=\lambda \phi(A)$ for every $\lambda>0$. Therefore,
$$\det(\phi(A)+ \lambda\phi(B)) =\det(\phi (A)+ \phi(\lambda B))=\det(A+\lambda B)$$
for every $\lambda >0$. Hence for $0<t\leq 1$, $A,B\in P_n$,
\begin{align*}
\det(tA+(1-t)B)   &   = t^n\det\left(A+\left(\frac{1-t}{t}\right)B\right)   \\
& = t^n\det\left(\phi(A)+\frac{1-t}{t}\phi(B)\right)=\det(t\phi(A)+(1-t)\phi(B)).
\end{align*}  \qedhere    \end{proof}

Next, we recall a useful lemma which shows that the derivative of the determinant can be written with the help of trace.

\begin{lemma}[{\cite[Main Theorem]{Go72}}]\label{diff}
Let $U$ be an open subset of the reals. Let $A(t)$ be a differentiable matrix valued function on $U$. Then $d(t)\equiv \det A(t)$ is differentiable and satisfies
\begin{equation*}
d_t(t)=\tr[Adj A(t)\cdot A_t(t)], \ t\in U.
\end{equation*}
\end{lemma}

In the following lemma we collect some trace-equalities that are equivalent.

\begin{lemma}\label{traceAB2}
Let $\phi: M_n\rightarrow M_n$ be an unital linear map and $\phi(P_n)\subseteq P_n$. Then the following statements are equivalent.


\begin{itemize}
\item[(1)]  There exists a unitary matrix $U\in M_n$ such that $\phi$ is either
\begin{equation*}
\phi(A)=U^{*}AU   \quad  \quad  \forall \ A\in P_n \quad    \quad  \text{or}  \quad    \quad    \phi(A)=U^{*}A^{t}U  \quad \quad \forall \ A\in P_n.
\end{equation*}
\item[(2)] $\tr(\phi(A){\phi(B)}^{-1})=\tr(AB^{-1}), \quad \quad A,B \in P_n.$
\item[(3)] $\tr(\phi(A)^2)=\tr(A^2), \quad \quad A \in P_n.$
\item[(4)] $\tr(\phi(A)\phi(B))=\tr(AB), \quad \quad A,B \in P_n.$
\end{itemize}

\end{lemma}

\begin{proof} Since $\phi(P_n)\subseteq P_n$ and $\phi$ is linear on $M_n$, we know that $\phi$ is continuous and hence $\phi(M_n^{+})\subseteq \phi(M_n^{+})$, $\phi(H_n)\subseteq \phi(H_n)$.

(1) $\Rightarrow$ (2): It is trivial.

(2) $\Rightarrow$ (3): 
Take $B=I$ in the assumption. Then we have $\tr(A)=\tr(\phi(A))$ for any $A\in P_n$. Note that for any $A\in P_n$, by Choi's inequality ({\cite[Corollary 2.3]{Ch74}}), $\phi(A)^{-1}\leq \phi(A^{-1})$. That is, $\phi(A^{-1})^{-1}\leq \phi(A)$. Hence
$$\tr(A^2) =  \tr(A(A^{-1})^{-1})= \tr(\phi(A)\phi(A^{-1})^{-1})=   \tr(\phi(A)^{1/2}\phi(A^{-1})^{-1}\phi(A)^{1/2})\leq \tr(\phi(A)^2).$$
By Kadison's inequality ({\cite[p. 495]{Ka52}}), $\phi(A)^2\leq \phi(A^2)$. Therefore, we have
$$\tr(A^2) \leq \tr(\phi(A)^2)\leq \tr(\phi(A^2))=\tr(A^2).$$
This implies that  $\tr(\phi(A)^2)=\tr(A^2)$.

(3) $\Rightarrow$ (4): For any $A, B\in P_n$, $\tr((A+B)^2)=\tr(A^2)+\tr(B^2)+2\tr(AB)$ and $\tr((\phi(A)+\phi(B))^2)=\tr(\phi(A)^2)+\tr(\phi(B)^2)+2\tr(\phi(A)\phi(B))$. Hence by assumption, we have $\tr(AB)=\tr(\phi(A)\phi(B))$.

(4) $\Rightarrow$ (1): {Taking $B=I$ we have that $\tr(\phi(A))=\tr (A)$.} Since $\phi(M_n^{+})\subseteq \phi(M_n^{+})$, by Kadison's inequality, $\phi(A)^2\leq \phi(A^2)$ for any $A>0$. Hence $$\tr(A^2)=\tr(\phi(A)^2)\leq \tr(\phi(A^2))=\tr(A^2).$$
Thus $\tr(\phi(A^2)-\phi(A)^2)=0$. By the inequality $\phi(A^2)\geq \phi(A)^2$ again, this implies that
$\phi(A^2)=\phi(A)^2$ for any $A>0$. For any $A\in H_n$, there exists $\epsilon>0$ such that $A+\epsilon I>0$ and hence $\phi((A+\epsilon I)^2)=\phi(A+\epsilon I)^2$. This implies that $\phi(A^2)=\phi(A)^2$. Now for any $A=A_1+iA_2\in M_n$ for some $A_1, A_2\in H_n$, we have $\phi (A_1+A_2)^2=\phi((A_1+A_2)^2)=\phi(A_1^2+A_2^2+A_1A_2+A_2A_1)$. This implies that $\phi(A_1A_2)+\phi(A_2A_1)=\phi (A_1)\phi(A_2)+\phi (A_2)\phi(A_1)$. Hence $\phi((A_1+iA_2)^2)=\phi(A_1^2)-\phi (A_2^2)+i\phi(A_1A_2+A_2A_1)=\phi(A_1)^2-\phi (A_2)^2+i\phi (A_1)\phi(A_2)+i\phi (A_2)\phi(A_1) =(\phi(A_1)+i\phi(A_2))^2=\phi(A_1+iA_2)^2$. In addition, since $\phi(H_n)\subseteq \phi(H_n)$, we have $\phi(A^{*})=\phi(A_1)-i\phi(A_2)=(\phi(A_1)+i\phi(A_2))^*=\phi(A)^*$. Hence $\phi$ is a Jordan $*$-homomorphism.

Moreover, by the proof of Lemma \ref{homoPn}, we can see that the complex span of $\{\phi(S_i): 1\leq i\leq n^2\}$ {is} $M_n$, where $\{S_i: 1\leq i\leq n^2\}$ is a fixed maximal real linearly independent subset  of $P_n$. It means that $\phi$ is surjective and hence injective since $\phi$ is linear on $M_n$. Therefore, $\phi$ is a Jordan $*$-isomorphism on $M_n$. Hence $\phi(A)=UAU^*$ for all $A\in P_n$ or $\phi(A)=UA^{T}U^*$ for all $A\in P_n$.    \qedhere  \end{proof}

\begin{remark}\label{remark1}
The assumption ``$\phi$ is linear" of Lemma \ref{traceAB2} is necessary for the condition (3). In fact, we can define the map $\phi: M_n\rightarrow M_n$ by $\phi(A)=U(s)AU(s)^*$ for all $A\in M_n$ with $\|A\|=s$, where $U(s)$ is a unitary matrix depending on the norm of $A$. Thus it is easy to check that $\phi(P_n)\subseteq P_n$ and $\tr(\phi(A)^2)=\tr(A^2)$. However, $\phi$ is not linear and doesn't satisfy the condition (1) in Lemma \ref{traceAB2}.
\end{remark}

Now, we are in a position to prove Lemma \ref{traceAB3}.

\begin{proof}[Proof of Lemma \ref{traceAB3}] It is easy to see (1) $\Rightarrow$ (2).

(2) $\Rightarrow$ (3): By Lemma \ref{equality}, we have for any $A,B \in P_n, \ 0\leq t\leq 1$
$$\det(tA+(1-t)B)=\det(t\phi(A)+(1-t)\phi(B)).$$
Hence Lemma \ref{diff} implies that
\begin{align}
  & \frac{d}{dt}\det(tA+(1-t)B)|_{t=0^+} =\frac{d}{dt}\det(B+t(A-B))|_{t=0^+}\\
  = & \tr[(Adj(B)+t(A-B))(A-B)]|_{t=0^+}=\tr[Adj(B)(A-B)]\\
  = & \tr(A\cdot Adj(B))-\tr(B\cdot Adj(B))= \tr(A\cdot Adj(B))-\tr[\det B\cdot I_n]\\
  = & \tr(A\cdot Adj(B))-n\cdot\det B.    \label{5}
\end{align}
Similarly, we have
\begin{align}\label{6}
  \frac{d}{dt}\det(t\phi(A)+(1-t)\phi(B))|_{t=0^+} & = \tr(\phi(A)\cdot Adj(\phi(B)))-n\cdot\det \phi(B).
\end{align}
Choosing $A=B$ in \eqref{1}, then we obtain $\det(\phi(B))=\det(B)$ for all $B\in P_n$. Hence by \eqref{5}, \eqref{6} and assumption, the equality
\begin{align*}
  \tr(A\cdot Adj(B)) = \tr(\phi(A)\cdot Adj(\phi(B)))        \quad \quad A,B \in P_n
\end{align*}
holds. Hence for all $A, B\in P_n$, we have
\begin{align*}
  \tr(A\cdot B^{-1}) = \tr\left(A\cdot \frac{Adj(B)}{\det(B)}\right) = \tr\left(\phi(A)\cdot \frac{Adj(\phi(B))}{\det(\phi(B))}\right)=  \tr(\phi(A)\cdot {\phi(B)}^{-1}).
\end{align*}

(3) $\Rightarrow$ (4): By Lemma \ref{homoPn}, we can obtain that $\phi$ is additive and positive homogeneous, that is, $\phi(A+B)=\phi(A)+\phi(B)$ and $\phi(\lambda A)=\lambda \phi(A)$ for all $A, B\in P_n$, $\lambda >0$. Firstly, for any $A\geq 0$, define $\phi_1: M_n^{+} \rightarrow M_n^{+}$ by $\phi_1(A)=\lim\limits_{\epsilon\rightarrow 0^{+}}\phi(A+\epsilon I)$. Then we can see that $\phi_1$ is well-defined, $\phi_1(A+B)=\phi_1(A)+\phi_1(B)$ and $\phi_1(\lambda A)=\lambda \phi_1(A)$ for all $A, B\in M_n^{+}$, $\lambda\geq 0$. Secondly, for any $A\in M_n$, there exist $A_1, A_2, A_3, A_4\geq0$ such that $A=A_1-A_2+iA_3-iA_4$. Then we can define $\phi_2: M_n \rightarrow M_n$ by $\phi_2(A)=\phi_1(A_1)-\phi_1(A_2)+i\phi_1(A_3)-i\phi_1(A_4)$. Then one can check that $\phi_2$ is well-defined, linear, unital and $\phi_2(A)=\phi(A)$ for any $A\in P_n$.
Hence our assertion follows from (2)$\Rightarrow$ (4) of Lemma \ref{traceAB2}.

(4) $\Rightarrow$ (1): By applying Lemma \ref{homoPn} and the similar proof (3) $\Rightarrow$ (4) in this lemma, we can define a unital linear map $\phi_3: M_n \rightarrow M_n$ satisfying $\phi_3(A)=\phi(A)$ for any $A\in P_n$.
Hence by (4)$\Rightarrow$ (1) of Lemma \ref{traceAB2}, there exists a unitary matrix $U\in M_n$ such that $\phi$ is either
\begin{equation*}
\phi(A)=U^{*}AU   \quad  \quad  \forall \ A\in P_n \quad    \quad  \text{or}  \quad    \quad    \phi(A)=U^{*}A^{t}U  \quad \quad \forall \ A\in P_n.
\end{equation*}   \qedhere  \end{proof}

\section{Maps  preserving determinants of convex combinations on $S_n$}

We consider the $S_n$ case in this section. In 2004, Cao and Tang studied this kind of problem and obtained some results in {\cite[Theorems 1.1 and 1.2]{Ca04}}. Now we add the trace equality and reformulate the result as follows.

\begin{theorem}\label{det4}

Let $\phi: S_n\rightarrow S_n$ be a map and  $\alpha^n=\det \phi(I)$. Then the following statements are equivalent.
\begin{itemize}
\item[(1)] There exists an invertible matrix $P\in S_n$ with $(\det P)^2=1$ such that $\phi$ is of the form
\begin{equation*}
\phi(A)=\alpha PAP^{t}   \quad  \quad  \forall \ A\in S_n.
\end{equation*}
\item[(2)] $\phi$ sends invertible matrices in $S_n$ into invertible matrices in $S_n$ and
\begin{equation*}
 \tr(\phi(A){\phi(B)}^{-1})=\tr(AB^{-1})
 \end{equation*}
 for all $A\in S_n$ and invertible matrix $B\in S_n$.
\item[(3)]
\begin{equation*}
\det(t\phi(A)+(1-t)\phi(B))=\alpha^n\det(tA+(1-t)B), \quad \quad A,B \in S_n, \ t\in [0,1]
\end{equation*}
\item[(4)] $\phi$ is surjective and there exist two specific $t\in [0,1]$ such that
\begin{equation*}
\det(t\phi(A)+(1-t)\phi(B))=\alpha^n\det(tA+(1-t)B), \quad \quad A,B \in S_n.
\end{equation*}

\end{itemize}

\end{theorem}

The implications (1) $\Leftrightarrow$ (3) and (1) $\Leftrightarrow$ (4) are slight modifications of Cao and Tang's results in {\cite[Theorems 1.1 and 1.2]{Ca04}}. Here we give another equivalent condition (1) $\Leftrightarrow$ (2) which is not trivial to see.

\begin{proof} It is easy to check (1) $\Rightarrow$ (2). We only need to prove (2) $\Rightarrow$ (1).

(2) $\Rightarrow$ (1): Since $\phi(I)\in S_n$, there exists an invertible matrix $Q\in S_n$ such that $\phi(I)=QQ^t$. Define $\psi: S_n \rightarrow P_n$ by $\psi(A)=Q^{-1}\phi(A)(Q^{-1})^t$. Then we can see that $\psi$ is unital and $\psi$ sends invertible matrices into invertible matrices. Moreover, $\tr(\psi(A){\psi(B)}^{-1})=\tr(\phi(A){\phi(B)}^{-1})=\tr(AB^{-1})$ for $A \in S_n$ and invertible  $B\in S_n$. By Lemma \ref{traceSn4} (see below), there exists an orthogonal matrix $P\in M_n$ such that $\psi(A)=PAP^t$ for $A\in S_n$. Hence $\phi(A)=Q\psi(A)Q^t=\alpha RAR^t$, where $R=\alpha^{-1/2}QP$ is invertible with $(\det R)^2=\alpha^{-n}\det(QQ^t)\det(P)^2=\alpha^{-n}\det\phi(I)=1$.  \qedhere  \end{proof}

For the proof of Theorem \ref{det4}, we should apply the following lemma.

\begin{lemma}\label{traceSn4}
Let $\phi: S_n\rightarrow S_n$ be an unital map sending invertible matrices into invertible matrices. Then the following statements are equivalent.

\begin{itemize}
\item[(1)]  There exists an orthogonal matrix $P\in M_n$ such that $\phi$ is of the form
\begin{equation*}
\phi(A)=PAP^{t}   \quad  \quad  \forall \ A\in S_n.
\end{equation*}
\item[(2)] $\tr(\phi(A){\phi(B)}^{-1})=\tr(AB^{-1}), \quad \quad A \in S_n$ and invertible $B\in S_n$.
\item[(3)] $\tr(\phi(A)\phi(B)^k)=\tr(AB^k), \quad \quad A, B \in S_n$ and $k\in \mathbb{N}$.
\end{itemize}

\end{lemma}

The proof of (3) $\Rightarrow$ (1) in Lemma \ref{traceSn4} below follows from the analogous idea in Section 3.

\begin{proof}[Proof of Lemma \ref{traceSn4}]

(1) $\Rightarrow$ (2): It is trivial.

(2) $\Rightarrow$ (3): By Lemma \ref{homo}, $\phi$ is linear. By assumption and $\phi(I)=I$, for small enough $|\lambda|$ we have $\tr(A(I-\lambda B)^{-1})=\tr(\phi(A)(I-\lambda \phi(B))^{-1})$. Hence for any $A, B\in S_n$
$$\sum_{k}\tr(AB^k)\lambda^k=\sum_{k}\tr(\phi(A)\phi(B)^k)\lambda^k.$$
Thus $\tr(AB^k)=\tr(\phi(A)\phi(B)^k)$.

(3) $\Rightarrow$ (1): By Lemma \ref{homo} and $\phi(I)=I$, we have that $\phi$ is linear, bijective and $\tr(A)=\tr(\phi(A))$. Note that for any polynomial $p(z)=\sum_{k=1}a_kz^k$, by assumption,
$$\tr([p(\phi(A))-\phi(p(A))]\phi(B))=\sum_ka_k\tr([\phi(A)^k-\phi(A^k)]\phi(B))
=\sum_ka_k(\tr(A^kB)-\tr(A^kB))=0.$$
Since $\phi$ is surjective, by Lemma \ref{lma1}, we have $$p(\phi(A))=\phi(p(A)).$$
Let $D_{ij}=E_{ij}+E_{ji}$ for $i\neq j$. Then $\phi(E_{ii})^2=\phi(E_{ii}^2)=\phi(E_{ii})$, $\phi(D_{ij})^2=\phi(D_{ij}^2)=\phi(E_{ii})+\phi(E_{jj})$, and $\phi(D_{ij})^3=\phi(D_{ij}^3)=\phi((E_{ii}+E_{jj})D_{ij})=\phi(D_{ij})$. Since $\tr(\phi(E_{ii})^k)=\tr(E_{ii}^k)$ and  $\tr(\phi(D_{ij})^k)=\tr(D_{ij}^k)$ for all $k$, by \cite[p.44, Problem 12]{Horn85}, $\phi(E_{ii})$ and $E_{ii}$ (resp., $\phi(D_{ij})$ and $D_{ij}$ have the same eigenvalues counting multiplicities. Moreover, all of them are diagonalizable. Therefore, rank $\phi(E_{ii})=1$ and rank $\phi(D_{ij})=2$. Expanding this equality $\phi[(E_{ii}+E_{jj})^2]=[\phi(E_{ii}+E_{jj})]^2=[\phi(E_{ii})+\phi(E_{jj})]^2$, we have $\phi(E_{ii})+\phi(E_{jj})=\phi(E_{ii})^2+\phi(E_{jj})^2+\phi(E_{ii})\phi(E_{jj})+\phi(E_{jj})\phi(E_{ii})$. Hence $$\phi(E_{ii})\phi(E_{jj})=-\phi(E_{jj})\phi(E_{ii}).$$
{Multiplying by $\phi(E_{jj})$ in the both sides and using that $\phi(E_{jj})^2=\phi(E_{jj})$}, we have this equality $[I_n+\phi(E_{jj})]\phi(E_{ii})\phi(E_{jj})=0$. Since $I_n+\phi(E_{jj})$ is invertible, we can obtain $\phi(E_{ii})\phi(E_{jj})=0$. Thus by \cite[p.216, Problem 24]{Horn85}, there exists a complex orthogonal matrix $Q\in M_n$ such that $\phi(E_{ii})=Q^{t}E_{ii}Q$. Replacing $\phi(A)$ with $Q\phi(A)Q^{t}$, we may assume $\phi(E_{ii})=E_{ii}$. Hence $\phi(D_{ij})^2=E_{ii}+E_{jj}$ and $\phi(D_{ij})=\phi(D_{ij})^3=(E_{ii}+E_{jj})\phi(D_{ij})=\phi(D_{ij})(E_{ii}+E_{jj})$. Thus
$$[\phi(D_{ij})]_{pq}=0 \quad \quad \forall p\neq i, j \ \text{or} \ q\neq i, j.$$
On the other hand, for any $k$, $$[\phi(D_{ij})]_{kk}=\tr(E_{kk}\phi(D_{ij}))=\tr(\phi(E_{kk})\phi(D_{ij}))=\tr(E_{kk}D_{ij})=0.$$
{Since $\phi(D_{ij})$ is symmetric, the above calculations show that $\phi(D_{ij})=a(E_{ij}+E_{ji})$, $a\neq0$.}  
Hence $\phi(D_{ij})=D_{ij}$ or $-D_{ij}$ since $\phi(D_{ij})^2=E_{ii}+E_{jj}$. Suppose $n=2$ and $\phi(D_{12})=-D_{12}$. Replacing $\phi(A)$ with $\left(\begin{array}{cc}
        1 & 0 \\
        0 & -1 \\
\end{array}
\right)\phi(A)\left(
\begin{array}{cc}
        1 & 0 \\
        0 & -1 \\
\end{array} \right)$, we may assume $\phi(D_{12})=D_{12}$. If $n\geq3$, then similarly we may assume that $\phi(D_{1j})=D_{1j}$ for all $j\neq 1$. Suppose $\phi(D_{uv})=-D_{uv}$ for some $u, v\neq 1$ and $u\neq v$. Then one can see that $\tr\phi((D_{1u}+D_{1v}+D_{uv})^3)=\tr(D_{1u}+D_{1v}+D_{uv})^3=6$ and
$\tr(\phi(D_{1u}+D_{1v}+D_{uv}))^3=\tr(D_{1u}+D_{1v}-D_{uv})^3=-6$. It is a contradiction since $\tr\phi(A^3)=\tr(\phi (A))^3$ for $A\in S_n$. Hence $\phi(D_{ij})=D_{ij}$ for all $i\neq j$. Therefore, $\phi(A)=A$ for all $A\in S_n$. This completes the proof.     \qedhere  \end{proof}




\section{Maps  preserving determinants of convex combinations on $M_n$}

In this section we consider the corresponding case on $M_n$. Similarly, we rewrite Tan and Wang's  result (\cite{Tan03}) by supplementing the trace equality.

\begin{theorem}\label{det2}

Let $\phi: M_n\rightarrow M_n$ be a map and  $\alpha^n=\det \phi(I)$. Then the following statements are equivalent.
\begin{itemize}
\item[(1)] There exist invertible matrices $M, N\in M_n$ with $\det(MN)=1$ such that $\phi$ is either
\begin{equation*}
\phi(A)=\alpha MAN   \quad  \quad  \forall \ A\in M_n \quad    \quad  \text{or}  \quad    \quad    \phi(A)=\alpha MA^{t}N  \quad \quad \forall \ A\in M_n.
\end{equation*}
\item[(2)] $\phi$ sends invertible matrices into invertible matrices and
\begin{equation*}
\tr(\phi(A){\phi(B)}^{-1})=\tr(AB^{-1})
\end{equation*}
 for all $A\in M_n$ and invertible matrix $B\in M_n$.
\item[(3)]
\begin{equation*}
\det(t\phi(A)+(1-t)\phi(B))=\alpha^n\det(tA+(1-t)B), \quad \quad A,B \in M_n, \ t\in [0,1]
\end{equation*}
\item[(4)] $\phi$ is surjective and there exist two specific $t\in [0,1]$ such that
\begin{equation*}
\det(t\phi(A)+(1-t)\phi(B))=\alpha^n\det(tA+(1-t)B), \quad \quad A,B \in M_n.
\end{equation*}
\end{itemize}


\end{theorem}

\begin{proof} 

By Tan and Wang's result in {\cite[Theorems 1 and 2]{Tan03}}, one can show the implications (1) $\Leftrightarrow$ (3) and (1) $\Leftrightarrow$ (4). It is easy to check (1) $\Rightarrow$ (2) and hence we only need to prove (2) $\Rightarrow$ (1).

(2) $\Rightarrow$ (1):  Define $\psi: M_n \rightarrow M_n$ by $\psi(A)={\phi(I)}^{-1}\phi(A)$. Then we can see that $\psi$ is unital and $\psi$ sends invertible matrices into invertible matrices. Moreover, $\tr(\psi(A){\psi(B)}^{-1})=\tr(\phi(A){\phi(B)}^{-1})=\tr(AB^{-1})$ for $A \in M_n$ and invertible $B\in M_n$. By Lemma \ref{homo}, $\psi$ is bijective and linear.  Similar to the proof in Lemma \ref{traceSn4}, we have that $p(\psi(A))=\psi(p(A))$ for any $A\in M_n$, polynomial $p$. Then $\psi(E_{ii})^2=\psi(E_{ii})$ and $\psi(E_{ii})$ are diagonalizable with $\psi(E_{ii})\psi(E_{jj})=0$, rank $\psi(E_{ii})=1$. Hence there exists an invertible matrix $P\in M_n$ such that $\psi(E_{ii})=P^{-1}E_{ii}P$. Replacing $\psi(A)$ with $P\psi(A)P^{-1}$, we may assume $\psi(E_{ii})=E_{ii}$. Let $D_{ij}=E_{ij}+E_{ji}$ for $i\neq j$. Then $\psi(D_{ij})^2=E_{ii}+E_{jj}$. Similar to the proof in Lemma \ref{traceSn4}, we have $\psi(E_{ij})+\psi(E_{ji})=\psi(D_{ij})=aE_{ij}+\frac{1}{a}E_{ji}$. Let $S_{ij}=E_{ij}-E_{ji}$. Similarly, we have 
$\psi(E_{ij})-\psi(E_{ji})=\psi(S_{ij})=cE_{ij}-\frac{1}{c}E_{ji}$. From the part $(2)\Rightarrow (3)$ of the proof of Lemma \ref{traceSn4} we also get that $\tr(AB^k)=\tr(\psi(A)\psi(B)^k)$, $k\in\N$. In particular, $$\tr[(\psi(E_{ij})+\psi(E_{ji}))(\psi(E_{ij})-\psi(E_{ji}))]
=\tr[(E_{ij}+E_{ji})(E_{ij}-E_{ji})]=0,$$
we have $\frac{c}{a}-\frac{a}{c}=0$ and hence $c=a$ or $c=-a$. If $c=a$, then $\psi(E_{ij})=aE_{ij}$ and $\psi(E_{ji})=\frac{1}{a}E_{ji}$. If $c=-a$, then $\psi(E_{ij})=\frac{1}{a}E_{ji}$ and $\psi(E_{ji})=aE_{ij}$.

We claim that there are no pairwise distinct $i, j, k$ such that $\psi(E_{ij})=aE_{ij}$ but $\psi(E_{ik})=\frac{1}{b}E_{ki}$. Otherwise, $[\psi(E_{ij})+\psi(E_{ik})]^2=(aE_{ij}+\frac{1}{b}E_{ki})^2=\frac{a}{b}E_{kj}$ and $\psi[(E_{ij}+E_{ik})^2]=\psi(0)=0$. It leads to a contradiction. Similarly, there are no pairwise distinct $i, j, k$ such that $\psi(E_{ij})=\frac{1}{a}E_{ji}$ but $\psi(E_{kj})=bE_{kj}$. Therefore, either $\psi(E_{ij})\in \mathbb{C}^{\star} E_{ij}$ for all $i, j$ or $\psi(E_{ij})\in \mathbb{C}^{\star}E_{ji}$ for all $i, j$. Here $\mathbb{C}^{\star}=\mathbb{C}\setminus\{0\}$. Suppose $\psi(E_{ij})=a_{ij}E_{ij}$ for all $i, j$. Since $\psi[(E_{ij}+E_{j1})^2]=[\psi(E_{ij})+\psi(E_{1j})]^2$, we have $a_{i1}E_{i1}=(a_{ij}E_{ij}+a_{j1}E_{j1})^2=a_{ij}a_{j1}E_{i1}$. Hence $a_{ij}=\frac{a_{i1}}{a_{j1}}$. Let $D=\diag(a_{11}, a_{21}, \ldots, a_{n1})$. Then $\psi(A)=DAD^{-1}$ and hence $$\phi(A)=\alpha(\frac{1}{\alpha}\phi(I)D)AD^{-1}.$$
In addition, $\det(\frac{1}{\alpha}\phi(I)DD^{-1})=\frac{1}{\alpha^n}\det\phi(I)=1$.

Similarly, if $\psi(E_{ij})=b_{ij}E_{ji}$ for all $i, j$, then $\psi(A)=DA^tD^{-1}$ for some diagonal matrix $D$. Hence
$$\phi(A)=\alpha(\frac{1}{\alpha}\phi(I)D)A^tD^{-1}$$
with $\det(\frac{1}{\alpha}\phi(I)DD^{-1})=1$. Our assertion follows.    \qedhere  \end{proof}

\section{Maps  preserving determinants of convex combinations on $T_n$}

In this section we add the trace equality to the case on $T_n$ which have been studied by Tan and Wang (\cite{Tan03}).

\begin{theorem}\label{det0}
Let $\phi: T_n\rightarrow T_n$ be a map and  $\alpha^n=\det \phi(I)$. Then the following statements are equivalent.
\begin{itemize}
\item[(1)] There exist a permutation $\sigma$ of degree $n$ and $\lambda_1$, $\ldots$, $\lambda_n$ with $\prod_{i=1}^n\lambda_i=1$ such that for all $A \in T_n$, \begin{equation*}
[\phi(A)]_{ii}=\alpha\lambda_iA_{\sigma(i) \sigma(i)} \quad \quad \forall \ 1\leq i\leq n.
\end{equation*}
\item[(2)] $\phi$ sends invertible matrices into invertible matrices and
\begin{equation*}
\tr(\phi(A){\phi(B)}^{-1})=\tr(AB^{-1})
\end{equation*}
 for all $A\in T_n$ and invertible matrix $B\in T_n$.
\item[(3)]
\begin{equation*}
\det(t\phi(A)+(1-t)\phi(B))=\alpha^n\det(tA+(1-t)B), \quad \quad A,B \in T_n, \ t\in [0,1]
\end{equation*}
\item[(4)] $\phi$ is surjective and there exist two specific $t\in [0,1]$ such that
\begin{equation*}
\det(t\phi(A)+(1-t)\phi(B))=\alpha^n\det(tA+(1-t)B), \quad \quad A,B \in T_n.
\end{equation*}
\end{itemize}

\end{theorem}

\begin{remark}
Following the similar proofs in {\cite[Theorems 1' and 2']{Tan03}}, and Lemma \ref{Tn2} (see below), one can show that Theorem \ref{det0} holds if $T_n$ is replaced by $D_n$ and $\phi: T_n \rightarrow T_n$ is replaced by $\phi: D_n \rightarrow D_n$.
\end{remark}

To prove Theorem \ref{det0}, we need the following lemma.

\begin{lemma}\label{Tn2}
Let $\phi: T_n\rightarrow T_n$ be an unital map sending invertible matrices into invertible matrices. Suppose $\tr(\phi(A){\phi(B)}^{-1})=\tr(AB^{-1})$ for $A \in T_n$ and invertible $B\in T_n$. Then there exists a permutation $\sigma$ of degree $n$ such that for all $A \in T_n$,
 \begin{equation*}
[\phi(A)]_{ii}=A_{\sigma(i) \sigma(i)} \quad \quad \forall \ 1\leq i\leq n.
\end{equation*}
\end{lemma}

\begin{proof}
As the idea in {\cite[Theorems 1' and 2']{Tan03}}, we can define $\psi: T_n \rightarrow D_n$ by $\psi(A)=\diag\phi(A)$, where $\diag A$ is the diagonal matrix having the same diagonal entries as $A\in T_n$. Then it is easily to see that $\psi$ is an unital map sending invertible matrices into invertible matrices and $\psi$ satisfies $ \tr(\psi(A){\psi(B)}^{-1})=\tr(AB^{-1})$ for $A \in T_n$ and invertible $B\in T_n$. By Lemma \ref{homo}, $\psi$ is linear. Similar to the proof in Lemma \ref{traceSn4}, we have that $\tr(\psi(A){\psi(B)}^k)=\tr(AB^k)$ for $A, B \in T_n$. Thus $\tr(\psi(A)^k)=\tr(A^k)$ for all $k$ and hence  by \cite[p.44, Problem 12]{Horn85}, $\psi(A)$ and $A$ have the same eigenvalues counting multiplicities. Since $A\in T_n, \psi(A)\in D_n$, there exists a permutation $\sigma$ of degree $n$ such that $\psi(A)=\diag(A_{\sigma(1) \sigma(1)}, \ldots, A_{\sigma(n) \sigma(n)})$. Therefore, $[\phi(A)]_{ii}=A_{\sigma(i) \sigma(i)}$ for all $i$.  \qedhere  \end{proof}

We can now present the proof of Theorem \ref{det0}.

\begin{proof}[Proof of Theorem \ref{det0}]
The implications (1) $\Leftrightarrow$ (3) and (1) $\Leftrightarrow$ (4) can be deduced by Tan and Wang's result in {\cite[Theorems 1' and 2']{Tan03}}. It is easy to check (1) $\Rightarrow$ (2). We only need to prove (2) $\Rightarrow$ (1).

(2) $\Rightarrow$ (1): Define $\psi: T_n \rightarrow T_n$ by { $\psi(A)=\phi(I)^{-1}\phi(A)$}. Then $\psi$ is an unital map sending invertible matrices into invertible matrices and $\psi$ satisfies $\tr(\psi(A){\psi(B)}^{-1})=\tr(AB^{-1})$ for $A \in T_n$ and invertible $B\in T_n$. Then by Lemma \ref{Tn2}, there exists a permutation $\sigma$ of degree $n$ such that for all $A \in T_n$, $[\psi(A)]_{ii}=A_{\sigma(i) \sigma(i)}$  for all $i$. Let $\lambda_i=\alpha^{-1}[\phi(I)]_{ii}$ for all $i$. Then $[\phi(A)]_{ii}=\alpha\lambda_iA_{\sigma(i) \sigma(i)}$. In addition, $\prod_{i=1}^n\lambda_i=\alpha^{-n}\det(\phi(I))=1$.
 \qedhere  \end{proof}

\section{Acknowledgement}
The authors would like to thank Professors Man-Duen Choi, Chi-Kwong Li, Lajos Moln\'ar, Tin-Yau Tam, and Ngai-Ching Wong for their useful comments. We are grateful for their important suggestions which helped to improve this paper. Liu was supported by the Taiwan MOST grant (104-2811-M-110-027). Szokol was supported by the ``Lend\" ulet'' Program (LP2012-46/2012) of the Hungarian Academy of Sciences.

\end{document}